\documentclass[a4paper,11pt]{article}

\usepackage{graphicx}
\usepackage{amssymb,amsmath,amsthm,times,mathrsfs,fullpage}
\usepackage{tikz,wrapfig,multirow}
\usetikzlibrary{patterns,arrows,decorations.pathreplacing}
\usepackage[pdftex]{hyperref}
\usepackage{epstopdf}
\hypersetup{plainpages=True, pdfstartview=FitH, bookmarksopen=true, colorlinks=true,linkcolor=blue,citecolor=blue}

\def\qed{{\quad\rule{1mm}{3mm}\,}}

\begin{document}

\pgfdeclarelayer{background}
\pgfdeclarelayer{foreground}
\pgfsetlayers{background,main,foreground}

\newtheorem{thm}{Theorem}
\newtheorem{cor}[thm]{Corollary}
\newtheorem{lmm}[thm]{Lemma}
\newtheorem{conj}[thm]{Conjecture}
\newtheorem{pro}[thm]{Proposition}
\newtheorem{Def}[thm]{Definition}
\theoremstyle{remark}\newtheorem{Rem}{Remark}

\title{Counting Phylogenetic Networks with Few Reticulation Vertices: A Second Approach}
\author{Michael Fuchs\\
    Department of Mathematical Sciences\\
    National Chengchi University\\
    Taipei 116\\
    Taiwan \and
    En-Yu Huang\\
    Department of Mathematical Sciences\\
    National Chengchi University\\
    Taipei 116\\
    Taiwan \and
    Guan-Ru Yu\\
    Department of Mathematics\\
    National Kaohsiung Normal University\\
    Kaohsiung 824\\
    Taiwan}

\maketitle

\begin{abstract}
Tree-child networks, one of the prominent network classes in phylogenetics, have been introduced for the purpose of modeling reticulate evolution. Recently, the first author together with Gittenberger and Mansouri (2019) showed that the number ${\rm TC}_{\ell,k}$ of tree-child networks with $\ell$ leaves and $k$ reticulation vertices has the first-order asymptotics
\[
{\rm TC}_{\ell,k}\sim c_k\left(\frac{2}{e}\right)^{\ell}\ell^{\ell+2k-1},\qquad (\ell\rightarrow\infty).
\]
Moreover, they also computed $c_k$ for $k=1,2,$ and $3$. In this short note, we give a second approach to the above result which is based on a recent (algorithmic) approach for the counting of tree-child networks due to Cardona and Zhang (2020). This second approach is also capable of giving a simple, closed-form expression for $c_k$, namely, $c_k=2^{k-1}\sqrt{2}/k!$ for all $k\geq 0$.
\end{abstract}

\section{Introduction and Results}\label{intro}

Over the last decades, phylogenetic networks have become more and more popular as models in evolutionary biology. As a result, apart from biological and algorithmic studies, recently also combinatorial and probabilistic studies have been undertaken for many of the fundamental classes of phylogenetic networks; e.g., see \cite{BiLaSt,BoGaMa,DiSeWe,FuGiMa1,FuGiMa2,FuYuZh1,FuYuZh2,St}.

This short note is concerned with tree-child networks which have been introduced by Cardona et al. in \cite{CaRoVa}. We recall how they are defined. First, a (bifurcating and bicombining) rooted {\it phylogenetic network} on $\ell$ leaves is defined as a directed acyclic graph (DAG) with no double edges which has the following vertices:
\begin{itemize}
\item[(i)] a unique {\it root} with indegree $0$ and outdegree $1$;
\item[(ii)] {\it leaves} with indegree $1$ and outdegree $0$ which are bijectively labeled by $\{1,\ldots,\ell\}$;
\item[(iii)] {\it tree vertices} with indegree $1$ and outdegree $2$;
\item[(iv)] {\it reticulation vertices} with indegree $2$ and outdegree $1$.
\end{itemize}
Such a rooted phylogenetic network is called a {\it tree-child network} if every non-leaf vertex has at least one child which is not a reticulation vertex, i.e., (a) a reticulation vertex is followed by a tree vertex or a leaf; (b) not both of the children of a tree vertex are reticulation vertices; and (c) the root is followed by a tree vertex or a leaf; see Figure~\ref{tc-ex}-(a) for an example.

Denote by ${\rm TC}_{\ell,k}$ the number of tree-child networks with $\ell$ leaves and $k$ reticulation vertices. Note that if $k=0$, then this number counts phylogenetic trees (i.e., rooted, binary trees with $\ell$ labeled leaves); see \cite{OEIS}.

In \cite{FuGiMa1}, the authors obtained the following first-order asymptotics of ${\rm TC}_{\ell,k}$: for any fixed $k$, there exists a constant $c_k>0$ such that, as $\ell\rightarrow\infty$,
\begin{equation}\label{asymp-exp}
{\rm TC}_{\ell,k}\sim c_k\left(\frac{2}{e}\right)^{\ell}\ell^{\ell+2k-1};
\end{equation}
see also \cite{FuGiMa2} for some corrections of the proof of \cite{FuGiMa1}. Moreover, the authors also computed $c_k$ for small values of $k$ and obtained that
\[
c_1=\sqrt{2},\qquad c_2=\sqrt{2},\qquad c_3=\frac{2\sqrt{2}}{3}.
\]
(For $k=0$, one easily finds $c_0=\sqrt{2}/2$; see for instance the introduction of \cite{DiSeWe}.) The method of \cite{FuGiMa1} yields $c_k$ also for larger values of $k$, however, the computation becomes more and more cumbersome; see the remark at the end of Section 3 in \cite{FuGiMa1}, where it was claimed that there is probably no closed-form expression for $c_k$. This claim turns out to be wrong; see our main theorem below.

The purpose of this note is to give a second proof of (\ref{asymp-exp}) which is based on the recent algorithmic approach to the counting of tree-child networks due to Cardona and Zhang \cite{CaZh}. Moreover, this approach is capable of giving a simple, closed-form expression for $c_k$ for all $k\geq 0$.

\begin{thm}\label{main-result}
For the number ${\rm TC}_{\ell,k}$ of tree-child networks with $\ell$ leaves and $k$ reticulation vertices, for fixed $k$ and $\ell\rightarrow\infty$,
\[
{\rm TC}_{\ell,k}\sim \frac{2^{k-1}\sqrt{2}}{k!}\left(\frac{2}{e}\right)^{\ell}\ell^{\ell+2k-1}.
\]
\end{thm}

A second class of phylogenetic networks which was treated in \cite{FuGiMa1} was the class of normal networks. Here, a tree-child network is called a {\it normal network} if the two parents of each reticulation vertex are incomparable with respect to the ancestor-descendant relation, e.g., the tree-child network in Figure~\ref{tc-ex}-(a) is not normal.

It was shown in \cite{FuGiMa1} that, for fixed $k$ and $\ell\rightarrow\infty$, we have ${\rm TC}_{\ell,k}\sim N_{\ell,k}$ where $N_{\ell,k}$ denotes the number of normal networks with $\ell$ leaves and $k$ reticulation vertices. Thus, we have the following corollary.

\begin{cor}
For the number $N_{\ell,k}$ of normal networks with $\ell$ leaves and $k$ reticulation vertices, for fixed $k$ and $\ell\rightarrow\infty$,
\[
N_{\ell,k}\sim \frac{2^{k-1}\sqrt{2}}{k!}\left(\frac{2}{e}\right)^{\ell}\ell^{\ell+2k-1}.
\]
\end{cor}

\begin{Rem}
Instead of considering phylogenetic networks with just the leaves labeled, some authors also considered phylogenetic networks with all non-root vertices bijectively labeled; see \cite{DiSeWe,FuGiMa1,FuGiMa2}.

Denote by $\widehat{{\rm TC}}_{n,k}$ (resp. $\widehat{N}_{n,k}$) the number of such tree-child (resp. normal) networks with $n$ non-root vertices and $k$ reticulation vertices. This number is closely related to ${\rm TC}_{\ell,k}$ (resp. ${\rm N}_{\ell,k}$); see Section 5 in \cite{FuGiMa1}. From this relationship and the above results, we obtain that for fixed $k$ and $n\rightarrow\infty$,

\begin{equation}\label{cor-node}
\widehat{{\rm TC}}_{n,k}\sim\widehat{N}_{n,k}\sim\frac{\sqrt{2}}{4^k k!}\left(1-(-1)^n\right)\left(\frac{\sqrt{2}}{e}\right)^n n^{n+2k-1}.
\end{equation}
\end{Rem}

We end the introduction by giving a short outline of the structure of this note. In the next section, we will recall the approach from \cite{CaZh} which was used to (a) compute values of $\mathrm{TC}_{\ell,k}$ for small $\ell$ and $k$ and (b) obtain formulas for all $\ell$ and $k=1$ and $k=2$. In fact, this approach is also useful for obtaining asymptotic results as will be shown in Section~\ref{aa} which contains the proof of our main result. Then, in a last section, we will give a brief comparison of the asymptotic approach introduced in this paper with the one from the previous publications \cite{FuGiMa1,FuGiMa2}.

\section{The Method of Cardona and Zhang}\label{meth-CaZh}

\begin{figure}
\begin{center}
\includegraphics[scale=0.68]{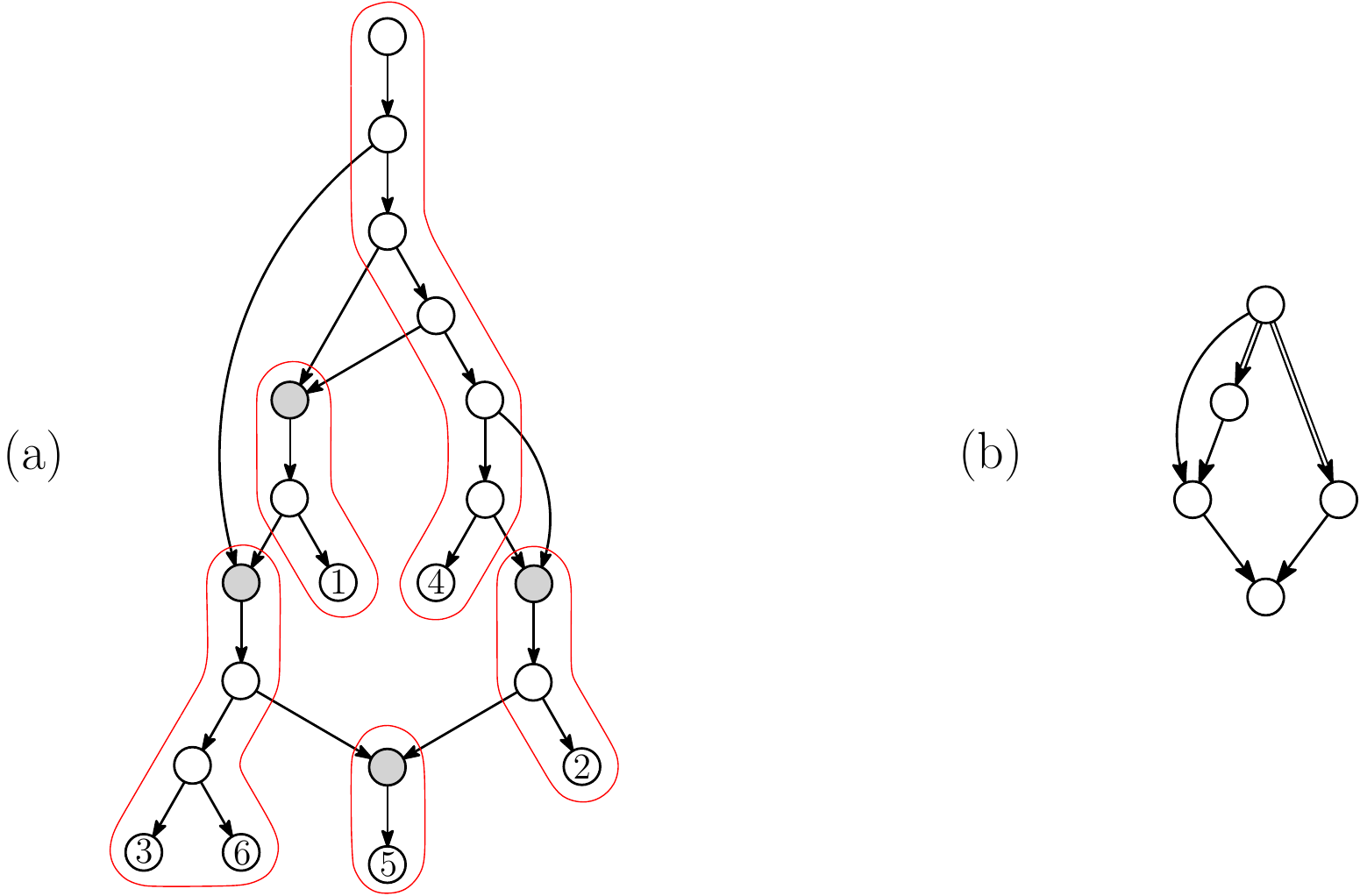}
\end{center}
\caption{(a) A tree-child network (which is not normal) with $6$ leaves and $4$ reticulation vertices (in gray). The encircled trees are obtained by deleting incoming edges of reticulation vertices; suppressing vertices of indegree and outdegree $1$ gives the tree components. (b) The corresponding component graph.}\label{tc-ex}
\end{figure}

One of the goals of \cite{CaZh} was to obtain a formula for $\mathrm{TC}_{\ell,k}$ which can be utilized to (algorithmically) compute these numbers for small values of $\ell$ and $k$. For this purpose, the authors in \cite{CaZh} proposed the notion of a \textit{component graph} whose definition we will recall next.

First, for a given tree-child network $N$ with $\ell$ leaves and $k$ reticulation vertices, deleting all incoming edges of reticulation vertices and suppressing resulting vertices with indegree and outdegree $1$ gives a forest consisting of $k+1$ trees each of which contains at least one of the $\ell$ leaves. (The latter follows from the tree-child property.) The components of this forest are called \textit{tree components}. The component graph of the tree-child network $N$ is then defined as follows: its vertex set is the set of the tree components of $N$ and edges between vertices mean that the tree components are connected via the deleted edges; see Figure~\ref{tc-ex}-(b) for an example.

Note that this definition implies that the set of component graphs is the set of rooted DAGs where double edges are allowed and each non-root vertex has indegree exactly equal to two. It is an easy exercise to show that each such DAG has at least one double edge starting from the root (two in the example from Figure~\ref{tc-ex}-(b)). Moreover, these DAGs can also be counted; see Theorem~15 in \cite{CaZh} for a formula for the number of component graphs with $m$ vertices.

All tree-child networks with $\ell$ leaves and $k$ reticulation vertices can now be obtained as follows: first list all the component graphs with $k+1$ vertices (see Figure~8 and Figure~9 in \cite{CaZh} for the lists for $k=0,1,2,3,4$); then, for each fixed component graph with $k+1$ vertices, replace the vertices by phylogenetic trees $T_1,\ldots, T_{k+1}$ and add back the deleted edges (by choosing suitable edges, adding a vertex inside them and connecting this new vertex to the root of a tree component); finally, partition the set $\{1,\ldots,\ell\}$ into $k+1$ blocks with block sizes equal to the number of leaves of $T_1,\ldots, T_{k+1}$, respectively, and re-label the leaves of the phylogenetic trees with the labels from their corresponding blocks in an order-consistent way. This gives a formula for the number of tree-child networks with $\ell$ leaves and $k$ reticulation vertices; see Theorem 16 of \cite{CaZh}.

A different way (which was also used in \cite{CaZh}) to perform the above procedure is to work with \textit{one-component tree-child networks} which are tree-child networks with each reticulation vertex  immediately followed by a leaf. (The name comes from the fact that one-component tree-child networks have only one non-trivial tree-component.)

The following formula for the number of one-component tree-child networks was proved in \cite{CaZh}.

\begin{lmm}[Theorem 13 in \cite{CaZh}]\label{number-o}
Let $\mathrm{O}_{\ell,k}$ denote the number of one-component tree-child networks with $\ell$ leaves and $k$ reticulation vertices where the labels of the leaves below the reticulation vertices are from the set $\{1,\ldots,k\}$. Then,
\[
\mathrm{O}_{\ell,k}=\frac{(2\ell-2)!}{2^{\ell-1}(\ell-k-1)!},\qquad (k\leq\ell-1).
\]
\end{lmm}
\begin{Rem}\label{rem}
\begin{itemize}
\item[(i)] For $k=0$, this gives the well-known formula $(2\ell-3)!!=(2\ell-3)\cdots 5\cdot 3\cdot 1$ for the number of phylogenetic trees on $\ell$ leaves.
\item[(ii)] Note that for any one-component tree-child network, the number of reticulation vertices is strictly smaller than the number of leaves. (This more generally holds for any tree-child network.)
\end{itemize}
\end{Rem}

Using these networks, we can now give a second construction procedure for building all tree-child networks with $\ell$ leaves and $k$ reticulation vertices from component graphs with $k+1$ vertices.

First, we will explain a reduction procedure for a component graph $C$ with $k+1$ vertices. We assume that $C$ has $1\leq t\leq k$ double edges starting from the root. First, remove all the single edges starting from the root and merge vertices with only one incoming (simple) edge with their parents. The resulting structure consists of a root with $t$ double edges which is on top of a DAG $\tilde{C}$ which has $t$ roots and again every non-root vertex has indegree exactly equal to two; see Step (i) of Figure~\ref{red-comp} for examples. In the next step, remove edges of $\tilde{C}$ until it decomposes into $t$ connected DAGs each of which is rooted at one of the $t$ roots of $\tilde{C}$. Note that this is always possible: for instance for each non-root vertex in $\tilde{C}$ whose two incoming edges are not a double edge, pick one of the edges and remove it. Next, again merge vertices with only one incoming edge with their parents; see Step (ii) of Figure~\ref{red-comp} for examples. We call the resulting DAG after these two steps a {\it reduced component graph} of $C$: it consists of a root with $t$ outgoing double edges and the children of this root are the roots of DAGs $D_1,\ldots,D_t$ which are again component graphs with $\beta_1,\ldots,\beta_t$ non-root vertices. Set $m:=\beta_1+\cdots+\beta_t$ and note that (i) $t+m\leq k$ and (ii) exactly $k-t-m$ edges have been removed in the above two steps.

\begin{figure}
\begin{center}
\includegraphics[scale=0.68]{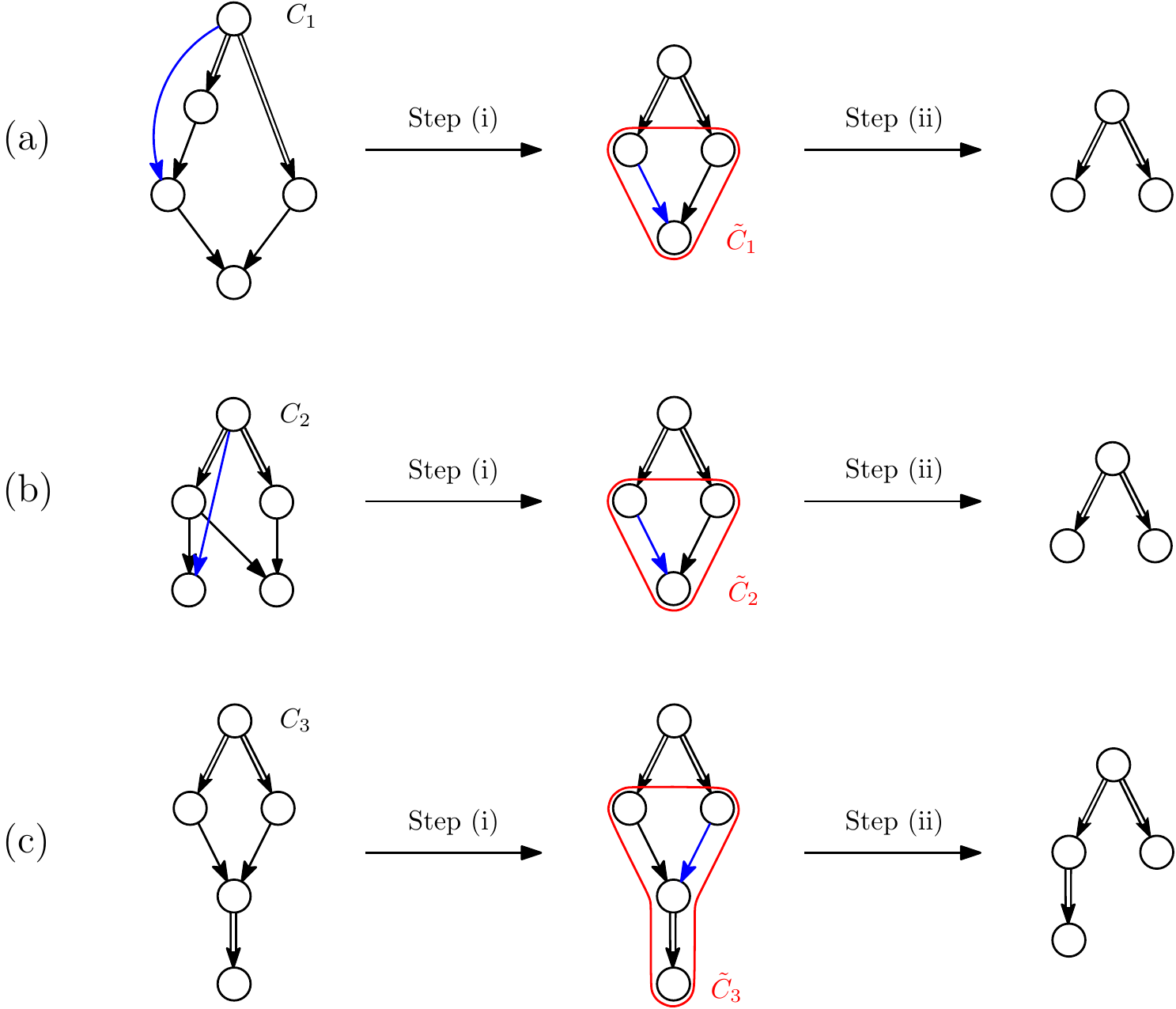}
\end{center}
\caption{Three component graphs and their reductions constructed via the two steps from Section~\ref{meth-CaZh} (the edges which are removed are highlighted in blue). The reduced component graphs of (a) and (b) are unique and coincide; $t=2$ and $\beta_1=\beta_2=0$. In (c), depending on which single edge is removed in Step (ii), two reduced component graphs are possible; for the one displayed, $t=2, \beta_1=1$ and $\beta_2=0$.}\label{red-comp}
\end{figure}

It is important to note that more than one possible choice of a reduced component graph might be possible for a component graph. Also, different component graphs might yield the same reduced component graph; see Figure~\ref{red-comp}.

Now, in order to construct all tree-child networks with $\ell$ leaves and $k$ reticulation vertices, we start again from the component graphs which are first reduced. Then, for a reduced component graph, the root with its outgoing double edges is replaced by a one-component tree-child network $O$ with $t$ reticulation vertices which are followed by leaves which are labeled by $\{1,\ldots,t\}$ (notation is as above). Here, the incoming edges of the reticulation nodes correspond to the outgoing double edges and the leaves below them correspond to the roots of $D_1,\ldots,D_t$. Then, these leaves are replaced by tree-child networks $N_1,\ldots,N_t$ whose component graphs are given by $D_1,\ldots,D_t$. Next, the $k-t-m$ removed edges from the reduction steps above are re-attached (by choosing two suitable edges, adding a vertex inside them and connecting the added vertices). Finally, the set $\{1,\ldots,\ell\}$ is partitioned into $t+1$ blocks the first of which has size equal to the number of leaves of $O$ which are not below reticulation vertices and the remaining have sizes equal to the number of leaves of $N_1,\ldots,N_t$. Then, the leaves of the above constructed network are re-labeled with their corresponding blocks in an order-consistent way.

In order to give an example for the construction just described, consider the component graph in Figure~\ref{tc-ex}-(b) whose reduced component graph is displayed in Figure~\ref{red-comp}-(a). Note that $t=2$ and $D_1,D_2$ consist both of a single root (thus, $m=0$ and consequently $2$ edges have been removed). Now, we need to pick a one-component tree-child network with $2$ reticulation vertices and replace the two leaves below the reticulation vertices both by phylogenetic trees. Finally, we add an edge between the one-component tree-child network and one of the phylogenetic trees and an edge between the two phylogenetic trees. Then, re-labeling gives, e.g., the tree-child network in Figure~\ref{tc-ex}-(a).

This approach was used in \cite{CaZh} to find formulas for ${\rm TC}_{\ell,k}$ for $k=1,2$ and all $\ell$; see the proofs of the results in Section~5 in \cite{CaZh}. We will use this approach to prove our main result. The strategy is as follows: we will identify one (special) component graph which gives the main contribution; see Proposition~\ref{main-asymp} below. For all the other component graphs, we will give an upper bound on the number of corresponding tree-child networks arising from these component graphs from the above procedure (by allowing that networks might be counted multiple times) and this upper bound will turn out to be asymptotically negligible; see Proposition~\ref{neg-asymp} below.

Before going on, we recall a simple lemma which has been stated before and whose simple proof we leave to the interested reader. This lemma will be used to count the number of ways of adding back edges in the construction above. Recall that an edge of a phylogenetic network is called \textit{reticulation edge} if its tail is a reticulation vertex and \textit{tree edge} otherwise.

\begin{lmm}\label{tree-edges}
A phylogenetic network with $\ell$ leaves and $k$ reticulation vertices has $2k$ reticulation edges and $2\ell+k-1$ tree edges; consequently, the total number of edges is $2\ell+3k-1$.
\end{lmm}

%\pf Clearly, the number of reticulation edges is $2k$. As for tree edges, first by a simple counting argument, we obtain that the number of tree nodes is $\ell+k-1$. Thus, the number of tree nodes and leaves is $2\ell+k-1$ and these are exactly the terminal nodes of all tree edges.\qed

\section{Asymptotic Analysis}\label{aa}

In this section, we will prove our main result (Theorem~\ref{main-result}) for the number $\mathrm{TC}_{\ell,k}$ of tree-child networks with $\ell$ leaves and $k$ reticulation vertices. For this purpose, we will use the second construction procedure from the previous section. The main observation is that the main term of the asymptotics will arise from the tree-child networks that are constructed from the \textit{star component graph of size $k$} which is the component graph consisting of a root which has $k$ children; see Figure~\ref{comp-graphs}. This case will be treated in the first paragraph below. Then, in the second paragraph, we will finish our asymptotic analysis by showing that the contribution of all other component graphs is asymptotically negligible.

\paragraph{Star Component Graph.} Denote by $\mathrm{S}_{\ell,k}$ the number of tree-child networks with $\ell$ leaves and $k$ reticulation vertices that arise from the star component graph of size $k$. We have the following formula.

\begin{figure}
\begin{center}
\includegraphics[scale=0.75]{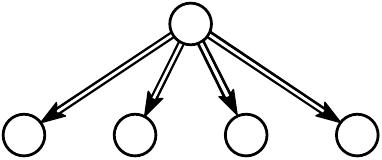}
\end{center}
\caption{The star component graph of size $4$.}\label{comp-graphs}
\end{figure}

\begin{lmm}
For $k\geq 1$, we have
\begin{equation}\label{S-l-k}
\mathrm{S}_{\ell,k}=\frac{\ell!}{2^{\ell-1}(k-1)!}\sum_{j=1}^{\ell-k}\frac{(2j+2k-2)!}{ j!(j-1)!}\cdot\frac{(2\ell-2j-k-1)!}{(\ell-j-k)!(\ell-j)!}.
\end{equation}
\end{lmm}
\begin{proof}
We use the second construction from the previous section. First, note that the reduced component graph of the star component graph of size $k$ is the graph itself. Thus, $t=k$ and $\beta_1=\cdots=\beta_k=0$. Consequently, we first need to pick a one-component tree-child network with $k$ reticulation vertices; then, we replace all the leaves below the reticulation vertices by phylogenetic trees; and finally, we re-label (in an order consistent way) all leaves of the resulting network such that only labels from the set $\{1,\ldots,\ell\}$ are used.

Based on this, we can give the following formula:
\begin{equation}\label{star-formula}
\mathrm{S}_{\ell,k}=\sum_{j=1}^{\ell-k}\binom{\ell}{j}\frac{(2j+2k-2)!}{2^{j+k-1}(j-1)!}\cdot\frac{1}{k!}(\ell-j)![z^{\ell-j}]T(z)^{k},
\end{equation}
where $T(z)$ is the exponential generating function of the number of phylogenetic trees (see Remark~\ref{rem}, (ii)) and $[z^n]f(z)$ denotes the $n$-th coefficient of the Maclaurin series of $f(z)$.

The factors inside the above sum are explained as follows: the factor $(2j+2k-2)!/(2^{j+k-1}(j-1)!)$ counts the number of one-component tree-child networks with $k$ reticulation vertices and $k+j$ leaves of which the ones with labels $\{1,\ldots,k\}$ are below the $k$ reticulation vertices; see Lemma~\ref{number-o}. The latter leaves are replaced by a set of $k$ phylogenetic trees whose number is given by $((\ell-j)![z^{\ell-j}]T(z)^{k})/k!$ since $T(z)^{k}$ enumerates sequences of $k$ phylogenetic trees and the factor $1/k!$ removes the order. Finally, the binomial coefficient takes care of all the possibilities of re-labeling the leaves.

Now, since we know from Lemma~\ref{number-o} (see also the remark after the lemma) that
\[
[z^{\ell}]T(z)=\frac{(2\ell-2)!}{2^{\ell-1}(\ell-1)!}=(2\ell-3)!!,
\]
a straightforward computation gives
\[
T(z)=\sum_{\ell\geq 1}(2\ell-3)!!\cdot\frac{z^{\ell}}{\ell!}=1-\sqrt{1-2z}
\]
which satisfies the equation $T(z)=z+T(z)^2/2$. (This equation could also be derived directly from the definition of phylogenetic trees and symbolic combinatorics; see Section 3.9 in \cite{FlSe}.)

Thus, from the Lagrange inversion formula (see, e.g., p126 in \cite{FlSe}),
\begin{align*}
[z^{\ell-j}]T(z)^k&=\frac{k}{(\ell-j)}[\omega^{\ell-j-k}](1-\omega/2)^{-(\ell-j)}\\
&=\frac{k}{(\ell-j)}2^{j+k-\ell}\binom{-(\ell-j)}{\ell-j-k}(-1)^{\ell-j-k}\\
&=\frac{k}{(\ell-j)}2^{j+k-\ell}\binom{2\ell-2j-k-1}{\ell-j-k}.
\end{align*}

Plugging this into (\ref{star-formula}) and re-arrangement gives the claimed result.
\end{proof}

We apply now the Laplace method (see, e.g., Section 4.7 in \cite{FlSe}) to the sum from the last lemma in order to obtain the asymptotics of $\mathrm{S}_{\ell,k}$.

\begin{lmm}
For $k\geq 1$, we have, as $\ell\rightarrow\infty$,
\[
\mathrm{S}_{\ell,k}\sim \frac{\sqrt{2}d_k}{2(k-1)!}\left(\frac{2}{e}\right)^{\ell}\ell^{\ell+2k-1},
\]
where
\[
d_k:=\sum_{j\geq 0}\frac{(2j+k-1)!}{j!(j+k)!}4^{-j}.
\]
\end{lmm}
\begin{proof}
It is sufficient to consider
\begin{align}
\Sigma_{\ell,k}&:=\sum_{j=1}^{\ell-k}\frac{(2j+2k-2)!}{ j!(j-1)!}\cdot\frac{(2\ell-2j-k-1)!}{(\ell-j-k)!(\ell-j)!}\nonumber\\
&=\sum_{j=0}^{\ell-k-1}\frac{(2\ell-2j-2)!}{(\ell-j-k)!(\ell-j-k-1)!}\cdot\frac{(2j+k-1)!}{j!(j+k)!}\label{laplace-sum}
\end{align}
since the asymptotics of the factor in front of the sum of (\ref{S-l-k}) is easily derived from Stirling's formula; see (\ref{asymp-fac}) below.

The main contribution to the sum (\ref{laplace-sum}) comes from the terms with small values of $j$. Thus, we expand the first term inside (\ref{laplace-sum}) which gives, as $\ell\rightarrow\infty$,
\[
\frac{(2\ell-2j-2)!}{(\ell-j-k)!(\ell-j-k-1)!}=\frac{1}{\sqrt{\pi}}4^{\ell-j-1}\ell^{2k-3/2}\left(1+{\mathcal O}\left(\frac{1+j^2}{\ell}\right)\right)
\]
uniformly in $j$ with $j=o(\sqrt{\ell})$. In addition note that
\[
\frac{(2\ell-2j-2)!}{(\ell-j-k)!(\ell-j-k-1)!}
\]
is a decreasing sequence in $j$ for all $k\geq 1$.

Thus, by a standard application of the Laplace method, as $\ell\rightarrow\infty$,
\[
\Sigma_{\ell,k}\sim\frac{d_k}{\sqrt{\pi}}4^{\ell-1}\ell^{2k-3/2}
\]
which multiplied with
\begin{equation}\label{asymp-fac}
\frac{\ell!}{2^{\ell-1}(k-1)!}\sim\frac{2\sqrt{2\pi}}{(k-1)!}\left(\frac{1}{2e}\right)^{\ell}\ell^{\ell+1/2},\qquad (\ell\rightarrow\infty).
\end{equation}
gives the claimed result.
\end{proof}

The final step is to simplify the constant $d_k$ from the last lemma.

\begin{lmm}
For $k\geq 1$, we have
\[
\sum_{j\geq 0}\frac{(2j+k-1)!}{j!(j+k)!}4^{-j}=\frac{2^k}{k}.
\]
\end{lmm}
\begin{proof}
Set
\[
A(z):=\sum_{j\geq 0}\frac{(2j+k-1)!}{(j+k)!j!}z^{j}.
\]
Our goal is to find $A(1/4)$. To this end, note that
\begin{align*}
[z^{j}]A(z)&=\frac{(2j+k-1)!}{(j+k)!j!}\\
&=\frac{1}{(j+k)}\binom{2j+k-1}{j}\\
&=\frac{1}{(j+k)}[\omega^{j}](1-\omega)^{-(j+k)}\\
&=\frac{1}{k}\cdot\frac{k}{(j+k)}[\omega^{j+k-k}](1-\omega)^{-(j+k)}\\
&=\frac{1}{k}[z^{j+k}]C(z)^{k}=\frac{1}{k}[z^j]\left(\frac{C(z)}{z}\right)^k,
\end{align*}
where $C(z)$ is a function which satisfies the equation $C(z)-C(z)^2=z$ and the second last equality follows by an application of the Lagrange inversion formula (in reversed direction). Thus,
\[
C(z)=\frac{1-\sqrt{1-4z}}{2}.
\]
and consequently
\[
A(z)=\frac{1}{k}\cdot\left(\frac{1-\sqrt{1-4z}}{2z}\right)^k
\]
which yields $A(1/4)=2^k/k$ as claimed.
\end{proof}

Combining the last two lemmas, we have the following asymptotic result for $S_{\ell,k}$.

\begin{pro}\label{main-asymp}
For the number $\mathrm{S}_{\ell,k}$ of tree-child networks with $\ell$ leaves and $k$ reticulation vertices arising from the star component graph of size $k$, as $\ell\rightarrow\infty$,
\[
\mathrm{S}_{\ell,k}\sim\frac{2^{k-1}\sqrt{2}}{k!}\left(\frac{2}{e}\right)^{\ell}\ell^{\ell+2k-1}.
\]
\end{pro}

\paragraph{Remaining Component Graphs.} Denote by $\mathrm{R}_{\ell,k}$ the number of tree-child networks arising from all non-star component graphs; see Figure~\ref{red-comp} for some examples. The proof of Theorem~\ref{main-result} will be finished by showing that $\mathrm{R}_{\ell,k}$ contributes asymptotically less than $\mathrm{S}_{\ell,k}$.

\begin{pro}\label{neg-asymp}
For the number $\mathrm{R}_{\ell,k}$ of tree-child networks with $\ell$ leaves and $k$ reticulation vertices arising from all component graphs except the star component graph of size $k$, we have
\[
\mathrm{R}_{\ell,k}={\mathcal O}\left(\left(\frac{2}{e}\right)^{\ell}\ell^{\ell+2k-3/2}\right).
\]
\end{pro}

For the proof, we will use induction on $k$. However, we first need two (simple) lemmas. The first gives an  upper bound for the number $\mathrm{O}_{\ell,k}$ from Lemma~\ref{number-o}.

\begin{lmm}\label{o-est}
We have
\[
\mathrm{O}_{\ell,k}={\mathcal O}\left(\left(\frac{2}{e}\right)^{\ell}\ell^{\ell+k-1}\right)
\]
\end{lmm}
\begin{proof}
This is a consequence of Stirling's formula.
\end{proof}

The second lemma gives a bound for certain sums.

\begin{lmm}\label{sum-est}
Let $\alpha_0,\ldots,\alpha_t$ be real numbers none of which equal to $-1$. Set
\[
s:=\#\{0\leq j\leq t\ :\ \alpha_j>-1\}
\]
Then,
\[
\sum_{\ell_0+\cdots+\ell_t=\ell}\ell_0^{\alpha_0}\cdots\ell_t^{\alpha_t}={\mathcal O}\left(\ell^{s-1+\sum_{\alpha_j>-1}\alpha_j}\right),
\]
where the sum runs over all positive integers $\ell_0,\ldots,\ell_t$.
\end{lmm}
\begin{proof}
We first assume that $0<s<t+1$. Then, we can write the sum as
\[
\sum_{\ell_0+\cdots+\ell_t=\ell}\ell_0^{\alpha_0}\cdots\ell_t^{\alpha_t}=\sum_{i=1}^{\ell-1}\left(\sum_{\sum_{\alpha_j<-1}\ell_j=\ell-i}\left(\prod_{\alpha_j<-1}\ell_j^{\alpha_j}\right)\right)
\left(\sum_{\sum_{\alpha_j>-1}\ell_j=i}\left(\prod_{\alpha_j>-1}\ell_j^{\alpha_j}\right)\right).
\]
We will start by estimating the two terms inside this sum.

For the first term, we have
\begin{equation}\label{est-1}
\sum_{\sum_{\alpha_j<-1}\ell_j=\ell-i}\left(\prod_{\alpha_j<-1}\ell_j^{\alpha_j}\right)={\mathcal O}\left((\ell-i)^{\alpha}\right),
\end{equation}
where $\alpha=\max\{\alpha_j\ :\ \alpha_j<-1\}$. This follows because at least one of the $\ell_j$'s with $\sum_{\alpha_j<-1}\ell_j=\ell-i$ is at least $(\ell-i)/(t+1-s)$ (giving the claimed upper bound) and the series $\sum_{\ell=1}^{\infty}\ell^{\beta}$ converges for all $\beta<-1$ (giving a constant upper bound for the remaining $\ell_j$'s).

For the second term, by approximating by an integral,
\begin{align}
\sum_{\sum_{\alpha_j>-1}\ell_j=i}\left(\prod_{\alpha_j>-1}\ell_j^{\alpha_j}\right)&={\mathcal O}\left(i^{s-1+\sum_{\alpha_j>-1}\alpha_j}\int_{\sum_{\alpha_j>-1}x_j=1}\left(\prod_{\alpha_j>-1}x_j^{\alpha_j}\right){\rm d}{\bf x}\right)\nonumber\\
&={\mathcal O}\left(i^{s-1+\sum_{\alpha_j>-1}\alpha_j}\right)\label{est-2}
\end{align}
where the integral is ${\mathcal O}(1)$ since it converges.

Finally, by combining the estimates of the two terms:
\begin{align*}
\sum_{\ell_0+\cdots+\ell_t=\ell}\ell_0^{\alpha_0}\cdots\ell_t^{\alpha_t}
&={\mathcal O}\left(\sum_{i=1}^{\ell-1}i^{s-1+\sum_{\alpha_j>-1}\alpha_j}(\ell-i)^{\alpha}\right)\\
&={\mathcal O}\left(\sum_{1\leq i\leq\ell/2}i^{s-1+\sum_{\alpha_j>-1}\alpha_j}(\ell-i)^{\alpha}+\sum_{\ell/2\leq i\leq\ell-1}i^{s-1+\sum_{\alpha_j>-1}\alpha_j}(\ell-i)^{\alpha}\right).
\end{align*}
For the first sum, we have
\begin{align*}
\sum_{1\leq i\leq\ell/2}i^{s-1+\sum_{\alpha_j>-1}\alpha_j}(\ell-i)^{\alpha}&={\mathcal O}\left(\ell^{\alpha}\sum_{1\leq i\leq\ell/2}i^{s-1+\sum_{\alpha_j>-1}\alpha_j}\right)\\
&={\mathcal O}\left(\ell^{\alpha+s+\sum_{\alpha_j>-1}\alpha_j}\int_{0}^{1/2}x^{s-1+\sum_{\alpha_j>-1}\alpha_j}{\rm d}x\right)\\
&={\mathcal O}\left(\ell^{\alpha+s+\sum_{\alpha_j>-1}\alpha_j}\right).
\end{align*}
For the second sum, we have
\[
\sum_{\ell/2\leq i\leq\ell-1}i^{s-1+\sum_{\alpha_j>-1}\alpha_j}(\ell-i)^{\alpha}={\mathcal O}\left(\ell^{s-1+\sum_{\alpha_j>-1}\alpha_j}\sum_{j\geq 1}j^{\alpha}\right)={\mathcal O}\left(\ell^{s-1+\sum_{\alpha_j>-1}\alpha_j}\right).
\]
Thus,
\[
\sum_{\ell_0+\cdots+\ell_t=\ell}\ell_0^{\alpha_0}\cdots\ell_t^{\alpha_t}={\mathcal O}\left(\ell^{\alpha+s+\sum_{\alpha_j>-1}\alpha_j}\right)+{\mathcal O}\left(\ell^{s-1+\sum_{\alpha_j>-1}\alpha_j}\right)={\mathcal O}\left(\ell^{s-1+\sum_{\alpha_j>-1}\alpha_j}\right)
\]
which is the claimed result for $0<s<t$.

For the missing cases $s=0$ and $s=t+1$, the result is already implied by (\ref{est-1}) and (\ref{est-2}), respectively. This concludes the proof.
\end{proof}

We are now ready to prove Proposition~\ref{neg-asymp}.

\vspace*{0.35cm}\noindent\textit{Proof of Proposition~\ref{neg-asymp}.} We use induction on $k$. The claim is trivial for $k=1$ since $\mathrm{R}_{\ell,1}=0$.

Next, assume that the claim holds for all $k'<k$. Note that because of the induction hypothesis and Proposition~\ref{main-asymp}, we have
\begin{equation}\label{ind-hyp}
\mathrm{TC}_{\ell,k'}={\mathcal O}\left(\left(\frac{2}{e}\right)^{\ell}\ell^{\ell+2k'-1}\right).
\end{equation}
We will now prove the claim for $k$.

Fix a non-star component graph with $k+1$ vertices and consider its reduced component graph. Using the same notation as in Section~\ref{meth-CaZh}, let $t$ be the number of double edges starting from the root of the reduced component graph and $m:=\beta_1+\cdots+\beta_t$ where $\beta_j$ is the number of non-root vertices of the DAG rooted at the $j$-th child of the root in the reduced component graph (note that there are no edges between the DAGs rooted at the children because they have been removed when constructing the reduced component graph). A crucial fact used below is that $t<k$ since $t=k$ only holds for the star component graph of size $k$ (and this component graph is not considered).

Now, using the second construction from the previous section, we obtain (up to a constant) the following upper bound for the number of tree-child networks arising from the fixed component graph
\[
U_{\ell,k}:=\sum_{\ell_0+\cdots+\ell_t=\ell}\binom{\ell}{\ell_0,\ldots,\ell_t}\mathrm{O}_{\ell_0+t,t}\mathrm{TC}_{\ell_1,\beta_1}\cdots\mathrm{TC}_{\ell_t,\beta_t}\ell_0^{\delta_0}\ell_1^{\delta_1}\cdots\ell_t^{\delta_t}
\]
since we pick a one-component tree-child network $O$ with $t$ reticulation vertices labeled $1$ to $t$ whose leaves are replaced by tree-child networks with $\beta_1,\ldots,\beta_t$ reticulation vertices. Moreover, the binomial coefficient inside the sum takes care of the re-labeling and $\ell_0^{\delta_0}\cdots\ell_t^{\delta_t}$ upper bounds (up to a constant) the number of ways of re-attaching the edges which where removed from the component graph (the ends of each edge can be assigned to vertices from that graph; $\delta_0$ and $\delta_j$ then count the number of ends contained in the networks counted by $\mathrm{O}_{\ell_0+t,t}$ and $\mathrm{TC}_{\ell_j,\beta_j}$, respectively); see Lemma~\ref{tree-edges}. Here, $\delta_0,\ldots,\delta_t$ are non-negative integers such that
\begin{equation}\label{rel}
\delta_0+\ldots+\delta_t=2(k-t-m)
\end{equation}
since exactly $k-t-m$ edges are re-attached. (Note that this is indeed just an upper bound since some networks might be constructed multiple times when re-attaching edges. Moreover, since the leaves of $O$ are replaced by any tree-child networks with $\beta_1,\ldots,\beta_t$ reticulation vertices, this formula also counts networks whose component graph might not be the component graph we started out with.)

In order to estimate the above sum, by Lemma~\ref{o-est}, (\ref{ind-hyp}) and Stirling's formula,
\begin{align*}
U_{\ell,k}&=\ell!\sum_{\ell_0+\cdots+\ell_t=\ell}\frac{\ell_0^{\delta_0}\mathrm{O}_{\ell_0+t,t}}{\ell_0!}\cdot\frac{\ell_1^{\delta_1}\mathrm{TC}_{\ell_1,\beta_1}}{\ell_1!}\cdots\frac{\ell_t^{\delta_t}\mathrm{TC}_{\ell_t,\beta_t}}{\ell_t!}\\
&={\mathcal O}\left(\left(\frac{1}{e}\right)^{\ell}\ell^{\ell+1/2}\sum_{\ell_0+\cdots+\ell_t=\ell}2^{\ell_0}\ell_0^{2t+\delta_0-3/2}\cdot2^{\ell_1}\ell_1^{2\beta_1+\delta_1-3/2}\cdots2^{\ell_t}\ell_t^{2\beta_t+\delta_t-3/2}\right)\\
&={\mathcal O}\left(\left(\frac{2}{e}\right)^{\ell}\ell^{\ell+1/2}\sum_{\ell_0+\cdots+\ell_t=\ell}\ell_0^{2t+\delta_0-3/2}\cdot\ell_1^{2\beta_1+\delta_1-3/2}\cdots\ell_t^{2\beta_t+\delta_t-3/2}\right).
\end{align*}
Note that $2t+\delta_0-3/2>-1$ (since $t\geq 1$) and $2\beta_j+\delta_j-3/2>-1$ if and only if $\beta_j>0$ or $\delta_j>0$ and at least one these $t$ numbers is $>-1$ since $t<k$; see (\ref{rel}). Set
\[
s':=\#\{1\leq j\leq t\ :\ \beta_j>0\ \text{or}\ \delta_j>0\}\geq 1.
\]
Then, by Lemma~\ref{sum-est}, where the $s$ in that lemma is $s'+1$, we have
\[
U_{\ell,k}={\mathcal O}\left(\left(\frac{2}{e}\right)^{\ell}\ell^{\ell+1/2}\ell^{s'+2t+\delta_0-3/2+\sum_{j=1}^{t}(2\beta_j+\delta_j)-3t/2+3(t-s')/2}\right)\\
={\mathcal O}\left(\left(\frac{2}{e}\right)^{\ell}\ell^{\ell+2k-1-s'/2}\right)
\]
which gives the required bound for $U_{\ell,k}$ since $s'\geq 1$.

Finally, summing over all non-star component graphs with $k+1$ vertices (whose number just depends on $k$) gives the claimed result. (Recall that we do not care about double-counting since we only need a bound which is smaller than the first-order asymptotics from Proposition~\ref{main-asymp}.)\qed

\section{Conclusion}

The main purpose of this short note was to provide a second approach to the counting of tree-child networks with few reticulation vertices. This second approach is based on a recent approach of Cardona and Zhang \cite{CaZh} which was devised for the purpose of algorithmically solving the counting problem for tree-child networks. We showed that this approach can be used to solve the asymptotic counting problem, too. Moreover, this approach is capable of yielding the multiplicative constant of the first-order asymptotics, something which was left open by the previous approach from \cite{FuGiMa1,FuGiMa2}.

We are going to end the paper by giving a brief comparison between the previous approach and the approach from this note. In the previous approach, tree-child networks where built from so-called \textit{sparsened skeletons}, whereas in the current note, we used the component graphs from \cite{CaZh}. As pointed out in \cite{FuGiMa1} the sparsened skeletons which give the main contribution to the asymptotics are all rooted, non-plane trees with $k$ leaves whose number does not permit an easy formula. Because of this, the authors from \cite{FuGiMa1} expected that the multiplicative constant in the first-order asymptotics has no simple form; see the end of Section 3 in \cite{FuGiMa1}. However, this turned out to be wrong as shown in this note. In fact, using component graphs simplifies the situation because the main contribution to the asymptotics just comes from a single component graph, namely the star component graph; see our analysis in Section~\ref{aa}.

Both approaches are also capable of giving exact formulas for small values of the number $k$ of reticulation vertices; for the approach from \cite{FuGiMa1} this was shown in \cite{FuGiMa2} and for the approach from \cite{CaZh}, such results were contained in that publication.

Overall our note shows that the approach from \cite{CaZh} seems to be more flexible than the approach from \cite{FuGiMa1} for the counting of tree-child networks with fixed $k$. However, the approach from \cite{FuGiMa1} has still one advantage over the one from \cite{CaZh}, namely, it also applies, \textit{mutatis mutandis}, to the counting of normal networks which is not the case for the approach from \cite{CaZh} (because it is not clear how to avoid the creation of edges that violate the normal condition). In fact, the approach from \cite{FuGiMa1} was used to solve a question concerning the counting of normal networks from \cite{CaZh}; see \cite{FuGiMa2}.

\section*{Acknowledgments}
We thank Apoorva Khare who after seeing (\ref{cor-node}) for $k=1,2,3$ during a talk of the first author at the \textit{6th India-Taiwan Conference on Combinatorics} in Varanasi suggested that the multiplicative constant might have the claimed form. We also thank the anonymous referees for their many insightful suggestions.

The first and second author acknowledge financial support by the Ministry of Science and Technology, Taiwan under the grants MOST-107-2115-M-009-010-MY2 and MOST-109-2115-M-004-003-MY2; the third author was partially supported by the grant MOST-110-2115-M-017-003-MY3.

\end{document}